\documentclass[a4paper]{article}

\usepackage{amsmath}
\usepackage{amssymb}
\usepackage{amsthm}
\usepackage{enumitem}
\usepackage{geometry}
\usepackage{tikz-cd}

\theoremstyle{definition}
\newtheorem{dfn}{Definition}[section]
\newtheorem{ex}[dfn]{Example}

\theoremstyle{plain}
\newtheorem{lem}[dfn]{Lemma}
\newtheorem{prop}[dfn]{Proposition}
\newtheorem{thm}[dfn]{Theorem}
\newtheorem{cor}[dfn]{Corollary}
\newtheorem*{thm*}{Theorem}

\theoremstyle{remark}
\newtheorem{rem}[dfn]{Remark}

\numberwithin{equation}{section}

\AtEndDocument{%
  \par
  \medskip
  \begin{tabular}{@{}l@{}}%
    Naoya Hiramae\\
    \textsc{Department of Mathematics, Kyoto University}\\
    Kitashirakawa Oiwake-cho, Sakyo-ku, 606-8502, Kyoto\\
    \textit{E-mail address}: \texttt{hiramae.naoya.58r@st.kyoto-u.ac.jp}
  \end{tabular}
  \par
  \medskip
  \begin{tabular}{@{}l@{}}%
    Yuta Kozakai\\
    \textsc{Department of Mathematics, Tokyo University of Science}\\
    1-3, Kagurazaka, Shinjuku-ku, Tokyo, 162-8601, Japan\\
    \textit{E-mail address}: \texttt{kozakai@rs.tus.ac.jp}
  \end{tabular}
  }
  
\title{$\tau$-Tilting finiteness of group algebras of semidirect products of abelian $p$-groups and abelian $p'$-groups}
\author{Naoya Hiramae \and Yuta Kozakai}
\date{}

\begin{document}

\maketitle

\renewcommand{\thefootnote}{\fnsymbol{footnote}}
\footnote[0]{\emph{Mathematics Subject Classification} (2020). 20C20, 16G10.}
\footnote[0]{\emph{Keywords.} Group algebras, $p$-hyperfocal subgroups, $\tau$-tilting finite algebras.}
\renewcommand{\thefootnote}{\arabic{footnote}}

\begin{abstract}
    Demonet, Iyama and Jasso introduced a new class of finite dimensional algebras, \textit{$\tau$-tilting finite algebras}. It was shown by Eisele, Janssens and Raedschelders that tame blocks of group algebras of finite groups are always $\tau$-tilting finite. Given the classical result that the representation type (representation finite, tame or wild) of blocks is determined by their defect groups, it is natural to ask what kinds of subgroups control $\tau$-tilting finiteness of group algebras or their blocks. In this paper, as a positive answer to this question, we demonstrate that $\tau$-tilting finiteness of a group algebra of a finite group $G$ is controlled by a \textit{$p$-hyperfocal subgroup} of $G$ under some assumptions on $G$.\\
    \indent We consider a group algebra of a finite group $P\rtimes H$ over an algebraically closed field of positive characteristic $p$, where $P$ is an abelian $p$-group and $H$ is an abelian $p'$-group acting on $P$, and show that $p$-hyperfocal subgroups determine $\tau$-tilting finiteness of the group algebras in this case.
\end{abstract}

\section{Introduction}
Throughout this paper, $k$ always denotes an algebraically closed field with characteristic $p>0$ and algebras mean finite dimensional algebras with a ground field $k$. Modules are always left and finitely generated.\\
\indent Demonet, Iyama and Jasso \cite{DIJ} showed that the following conditions on an algebra $\Lambda$ for certain finiteness are equivalent:
\begin{itemize}
    \item The number of isoclasses of basic support $\tau$-tilting $\Lambda$-modules is finite.
    \item The number of isoclasses of indecomposable $\tau$-rigid $\Lambda$-modules is finite.
    \item The number of isoclasses of bricks over $\Lambda$ is finite.
    \item Every torsion class in the module category over $\Lambda$ is functorially finite.
\end{itemize}
Algebras which satisfy one of the above equivalent conditions are said to be $\tau$-tilting finite. To date, many researchers have studied the $\tau$-tilting finite algebras (for example, \cite{Ad, AAC, AH, AHMW, AW, ALS, KK1, KK2, K, MS, MW, Mi, Mo, P, STV, W1, W2, W3, W4, Z}). In the context of the modular representation theory of finite groups, Eisele, Janssens and Raedschelders \cite{EJR} showed the remarkable result that tame blocks of groups algebras of finite groups are $\tau$-tilting finite, which does not hold outside blocks of group algebras: there exist $\tau$-tilting infinite algebras of tame type in general (for example, see \cite[Theorem 6.7]{AAC}). Given the classical result that the representation type of blocks is determined by their defect groups (see Theorem \ref{repblock}), it is natural to wonder whether we can describe conditions for blocks to be $\tau$-tilting finite in terms of their defect groups. However, it is impossible to determine whether given blocks are $\tau$-tilting finite by looking at defect groups alone (see \cite[Corollary 4]{EJR} for example). Therefore, we aim to find out what kinds of subgroups control $\tau$-tilting finiteness of blocks of group algebras of finite groups. We obtain a sufficient condition for a group algebra of a finite group $G$ to be $\tau$-tilting finite in terms of a \textit{$p$-hyperfocal subgroup}, that is, the intersection of a Sylow $p$-subgroup of $G$ and $O^p(G)$:\\

\begin{prop}[See Proposition \ref{tfhyp}]\label{itfhyp}
    Let $R$ be a $p$-hyperfocal subgroup of a finite group $G$. Then a group algebra $kG$ is $\tau$-tilting finite if one of the following holds:
    \begin{enumerate}
        \setlength{\parskip}{0cm}
        \setlength{\itemsep}{0cm}
        \item $R$ is cyclic.
        \item $p=2$ and $R$ is isomorphic to a dihedral, semidihedral or generalized quaternion group.\\
    \end{enumerate}
\end{prop}

\noindent We conjecture that the converse of Proposition \ref{itfhyp} is also true. In this paper, we see that the conjecture holds in a special case.\\
\indent We consider $\tau$-tilting finiteness of a group algebra of a group $G:=P\rtimes H$, where $P$ is a $p$-group and $H$ is an abelian $p'$-group acting on $P$. When $P$ is abelian, the quiver and relations of the group algebra $kG$ are completely known as in Proposition \ref{abelquiv} (see \cite{BKL1} and \cite{BKL2} for more general results). We show that $\tau$-tilting finiteness of $kG$ can be reduced to the case $C_P(H)=1$ by replacing $P$ with $[P,H]$ (see Corollary \ref{red}), where $C_P(H)$ denotes the centralizer of $H$ in $P$. Moreover, it turns out that $[P,H]$ is the same as the $p$-hyperfocal subgroup of $G$ (see Lemma \ref{hyp}). By finding a \textit{zigzag cycle} (see Definition $\ref{zigdfn}$) in the quiver of $kG$ in the reduced case, we show that the $p$-hyperfocal subgroup of $G$ determines $\tau$-tilting finiteness of $kG$ as follows:\\

\begin{thm}[See Theorem \ref{mainthm}]
    Let $P$ be an abelian $p$-group, $H$ an abelian $p'$-subgroup acting on $P$ and $G:=P\rtimes H$. Denote by $R$ the $p$-hyperfocal subgroup of $G$. Then $kG$ is $\tau$-tilting finite if and only if one of the following holds:
    \begin{enumerate}
        \setlength{\parskip}{0cm}
        \setlength{\itemsep}{0cm}
        \item $p=2$ and $R$ is trivial or isomorphic to the Klein four group.
        \item $p\geq3$ and $R$ is trivial or cyclic.
    \end{enumerate}
\end{thm}

\vspace{3mm}

\noindent We should remark that $R$ cannot be cyclic when $p=2$ (see Lemma \ref{p2}) and that the Klein four group is exactly the dihedral group of order $4$, which verifies that the converse of Proposition \ref{itfhyp} holds in this case.\\
\indent When $P$ is not abelian, we cannot use the above reduction. But assuming that $C_P(H)$ is contained in the Frattini subgroup $\Phi(P)$ of $P$, we can determine $\tau$-tilting finiteness of $kG$ by the $p$-hyperfocal subgroup of $G$, except in the case where $p=2$ and $P$ has rank $2$ (in this case, $kG$ can be $\tau$-tilting finite or infinite as in Remark \ref{p2rank2}).\\

\begin{thm}[See Corollary \ref{maincor}]
    Let $P$ be a $p$-group, $H$ an abelian $p'$-subgroup acting on $P$, $G:=P\rtimes H$ and $R$ the $p$-hyperfocal subgroup of $G$. Assume that $C_P(H)\leq \Phi(P)$. Then the following hold:
    \begin{enumerate}
        \setlength{\parskip}{0cm}
        \setlength{\itemsep}{0cm}
        \item When $p=2$ and $P$ has rank $\neq2$, $kG$ is $\tau$-tilting finite if and only if $R$ is trivial.
        \item When $p\geq3$, $kG$ is $\tau$-tilting finite if and only if $R$ has rank $\leq1$.
    \end{enumerate}
\end{thm}

\vspace{3mm}

\noindent\textbf{Notation.} Unless otherwise specified, a symbol $\otimes$ means a tensor product over $k$.\\
\indent For an algebra $\Lambda$, we denote the opposite algebra of $\Lambda$ by $\Lambda^{\mathrm{op}}$, the category of $\Lambda$-modules by $\Lambda\textrm{-mod}$, the full subcategory of $\Lambda\textrm{-mod}$ consisting of all projective $\Lambda$-modules by $\Lambda\textrm{-proj}$, the homotopy category of bounded complexes of projective $\Lambda$-modules by $K^b(\Lambda\textrm{-}\mathrm{proj})$, and the derived category of bounded complexes of $\Lambda$-modules by $D^b(\Lambda\textrm{-mod})$.\\
\indent For $M\in\Lambda\textrm{-mod}$, we denote the number of nonisomorphic indecomposable direct summands of $M$ by $|M|$, the $k$-dual of $M$ by $M^*$ and the Auslander-Reiten translate of $M$ by $\tau M$ (for example, see \cite{ARS}). For a complex $T$, we denote a shifted complex by $i$ degrees of $T$ by $T[i]$.\\

\section{Preliminary}

In this section, we first recall the basic materials of $\tau$-tilting theory and $\tau$-tilting finite algebras. In the last subsection, we will introduce a \textit{zigzag cycle}, a sequence of arrows in the Gabriel quiver which induces $\tau$-tilting infiniteness of algebras.\\
\indent Throughout this section, let $\Lambda$ be an algebra.\\

\subsection{$\tau$-Tilting theory}

Adachi, Iyama and Reiten introduced $\tau$-tilting theory and revealed the relationship among support $\tau$-tilting modules, two-term silting complexes and functorially finite torsion classes \cite{AIR}. To date, it has been found that support $\tau$-tilting modules are in bijection with other many objects such as left finite semibricks \cite{A1}, two-term simple-minded collections \cite{KY}, intermediate $t$-structures of length heart \cite{BY}, and more. In this subsection, we only describe the main results in \cite{AIR}.\\

\begin{dfn}
    A module $M\in\Lambda\textrm{-}\mathrm{mod}$ is a \textit{support $\tau$-tilting module} if $M$ satisfies the following conditions:
    \begin{enumerate}
        \setlength{\parskip}{0cm}
        \setlength{\itemsep}{0cm}
        \item $M$ is $\tau$-rigid, that is, $\mathrm{Hom}_\Lambda(M,\tau M)=0$.
        \item There exists $P\in\Lambda\textrm{-}\mathrm{proj}$ such that $\mathrm{Hom}_{\Lambda}(P,M)=0$ and $|P|+|M|=|\Lambda|$.
    \end{enumerate}
    When we specify the projective $\Lambda$-module $P$ in (b), we write a support $\tau$-tilting module $M$ as a pair $(M,P)$.\\
\end{dfn}

\begin{prop}[{\cite[Theorem 2.7]{AIR}}]
    The set of isoclasses of basic support $\tau$-tilting $\Lambda$-modules is a partially ordered set with respect to the following relation:
    $$M\geq M'\Leftrightarrow there\;exists\;a\;surjective\;homomorphism\;from\;a\;direct\;sum\;of\;copies\;of\;M\;to\;M'.$$
\end{prop}

\vspace{3mm}

\begin{dfn}
    A complex $T\in K^b(\Lambda\textrm{-}\mathrm{proj})$ is \textit{tilting} (resp. \textit{silting}) if $T$ satisfies the following conditions:
    \begin{enumerate}
        \setlength{\parskip}{0cm}
        \setlength{\itemsep}{0cm}
        \item $T$ is \textit{pretilting} (resp. \textit{presilting}), that is, $\mathrm{Hom}_{K^b(\Lambda\textrm{-}\mathrm{proj})}(T,T[i])=0$ for all integers $i\neq 0$ (resp. $i>0$).
        \item The full subcategory $\mathrm{add}\,T$ of $K^b(\Lambda\textrm{-}\mathrm{proj})$ consisting of all complexes isomorphic to direct sums of direct summands of $T$ generates $K^b(\Lambda\textrm{-}\mathrm{proj})$ as a triangulated category.\\
    \end{enumerate}
\end{dfn}

\begin{prop}[{\cite[Theorem 2.11]{AI}}]
    The set of isoclasses of basic silting complexes in $K^b(\Lambda\textrm{-}\mathrm{proj})$ is a partially ordered set with respect to the following relation:
    $$T\geq T'\Leftrightarrow\mathrm{Hom}_{K^b(\Lambda\textrm{-}\mathrm{proj})}(T,T'[i])=0\;for\;all\;integers\;i>0.$$
\end{prop}

\vspace{3mm}

\indent We say that a complex $T\in K^b(\Lambda\textrm{-}\mathrm{proj})$ is \textit{two-term} if its $i$-th term $T^i$ vanishes for all $i\neq-1,0$. We denote by $\mathrm{s}\tau\textrm{-}\mathrm{tilt}\,\Lambda$ the set of isoclasses of basic support $\tau$-tilting $\Lambda$-modules and by $2\textrm{-}\mathrm{silt}\,\Lambda$ the set of isoclasses of basic two-term silting complexes in $K^b(\Lambda\textrm{-}\mathrm{proj})$.\\

\begin{thm}[{\cite[Theorem 3.2]{AIR}}]
    There exists an isomorphism as partially ordered sets
    $$
    \begin{tikzcd}[row sep = 1mm]
        \mathrm{s}\tau\textrm{-}\mathrm{tilt}\,\Lambda \rar[leftrightarrow, "\sim"] \dar[phantom, "\rotatebox{90}{$\in$}"] & 2\textrm{-}\mathrm{silt}\,\Lambda, \dar[phantom, "\rotatebox{90}{$\in$}"]\\
        (M,P) \rar[mapsto] & (P^M_1\oplus P\xrightarrow{(f\;0)}P^M_0), \\
        \mathrm{Cok}\,g \rar[mapsfrom] & (P^{-1}\xrightarrow{g}P^0), 
    \end{tikzcd}
    $$
    where $P^M_1\xrightarrow{f}P^M_0\rightarrow M\rightarrow0$ is the minimal projective presentation of $M$.\\
\end{thm}

\begin{dfn}
    A full subcategory $\mathcal{T}$ of $\Lambda\textrm{-mod}$ is a \textit{torsion class} if it is closed under taking factor modules and extensions, that is, for any short exact sequence $0\rightarrow L\rightarrow M\rightarrow N\rightarrow 0$ in $\Lambda\textrm{-mod}$, the following hold:
    \begin{enumerate}
        \setlength{\parskip}{0cm}
        \setlength{\itemsep}{0cm}
            \item $M\in\mathcal{T}$ implies $N\in\mathcal{T}$.
            \item $L,N\in\mathcal{T}$ implies $M\in\mathcal{T}$.\\
    \end{enumerate}
\end{dfn}

\begin{dfn}
    A full subcategory $\mathcal{A}$ of $\Lambda\textrm{-mod}$ is \textit{functorially finite} if for any $M\in\Lambda\textrm{-}\mathrm{mod}$, there exist $X,\,Y\in\mathcal{A}$ and $f:M\rightarrow X,\,g:Y\rightarrow M\in\Lambda\textrm{-}\mathrm{mod}$ such that $-\circ f:\mathrm{Hom}_\Lambda(X,Z)\rightarrow\mathrm{Hom}_\Lambda(M,Z)$ and $g\circ -:\mathrm{Hom}_\Lambda(Z,Y)\rightarrow\mathrm{Hom}_\Lambda(Z,M)$ are both surjective for any $Z\in\mathcal{A}$.
    $$
    \begin{tikzcd}
        M \ar[rd,"\forall"'] \ar[rr,"f"] &  & X \ar[ld, dotted, "\exists"] &  & Z \ar[ld, dotted, "\exists"'] \ar[rd, "\forall"] & \\
         & Z &  & Y \ar[rr,"g"'] &  & M
    \end{tikzcd}
    $$
\end{dfn}

\vspace{3mm}

We denote by $\textrm{f-tors}\,\Lambda$ the set of functorially finite torsion classes in $\Lambda\textrm{-mod}$.\\

\begin{thm}[{\cite[Theorem 2.7]{AIR}}]
    There exists a bijection
    $$
    \begin{tikzcd}[row sep = 2mm]
        \mathrm{s}\tau\textrm{-}\mathrm{tilt}\,\Lambda \rar[rightarrow, "\sim"] \dar[phantom, "\rotatebox{90}{$\in$}"] & \mathrm{f\textrm{-}tors}\,\Lambda, \dar[phantom, "\rotatebox{90}{$\in$}"]\\
        M \rar[mapsto] & \mathrm{Fac}\,M,
    \end{tikzcd}
    $$
    where $\mathrm{Fac}\,M$ means a full subcategory of $\Lambda\mathrm{\textrm{-}mod}$ consisting of factor modules of a direct sum of copies of $M$.\\
\end{thm}

\subsection{$\tau$-Tilting finite algebras}

Demonet, Iyama and Jasso introduced a new class of algebras, \textit{$\tau$-tilting finite algebras}, and found that $\tau$-tilting finiteness is equivalent to other conditions for certain finiteness: brick finiteness and functorially finiteness of all the torsion classes \cite{DIJ}.\\

\begin{dfn}
    An algebra $\Lambda$ is \textit{$\tau$-tilting finite} if there exist only finitely many basic support $\tau$-tilting modules up to isomorphism, that is, $\#\textrm{s$\tau$-tilt}\,\Lambda<\infty$.\\
\end{dfn}

\begin{dfn}
    A module $M\in\Lambda\textrm{-mod}$ is a \textit{brick} if the endomorphism algebra $\mathrm{End}_\Lambda(M)$ is isomorphic to $k$.\\
\end{dfn}

\begin{thm}[{\cite[Theorems 3.8 and 4.2]{DIJ}}]\label{DIJ}
    For an algebra $\Lambda$, the following are equivalent:
    \begin{enumerate}
        \setlength{\parskip}{0cm}
        \setlength{\itemsep}{0cm}
        \item $\Lambda$ is $\tau$-tilting finite.
        \item The number of isoclasses of bricks in $\Lambda\textrm{-}\mathrm{mod}$ is finite.
        \item Every torsion class in $\Lambda\textrm{-}{\mathrm{mod}}$ is functorially finite.\\
    \end{enumerate}
\end{thm}

\indent In the context of modular representation theory of finite groups, Eisele, Janssens and Raedschelders showed the following result:\\

\begin{thm}[See subsection 6.2 in \cite{EJR}]\label{tameblock}
    Tame blocks of group algebras of finite groups are $\tau$-tilting finite.\\
\end{thm}

\noindent This property is peculiar to blocks of group algebras since there exist some examples of $\tau$-tilting infinite algebras of tame type which cannot be realized as blocks of group algebras (for example, see \cite[Theorem 6.7]{AAC}). By Theorem 
\ref{tameblock}, we have the following hierarchy of representation types of blocks of group algebras of finite groups:
$$
\begin{tikzcd}
    \textrm{representation finite blocks} \rar[Rightarrow] & \textrm{tame blocks} \rar[Rightarrow, "\textrm{Theorem} \ref{tameblock}"] &[10mm] \textrm{$\tau$-tilting finite blocks}.
\end{tikzcd}
$$
Note that the converse of each implication above does not hold.\\
\indent We can find that blocks are representation finite or tame by looking at their defect groups as the following:\\

\begin{thm}[See THEOREM in Introduction in {\cite{E}}]\label{repblock}
    Let $B$ be a block of a group algebra $kG$ of a finite group $G$ and $D$ a defect group of $B$. Then the following hold:
    \begin{enumerate}
        \setlength{\parskip}{0cm}
        \setlength{\itemsep}{0cm}
        \item $B$ is representation finite if and only if $D$ is cyclic.
        \item $B$ is representation infinite and tame if and only if $p=2$ and $D$ is isomorphic to a dihedral, semidihedral or generalized quaternion group.\\
    \end{enumerate}
\end{thm}

\noindent It is natural to wonder whether there exist conditions on defect groups for blocks to be $\tau$-tilting finite. However, blocks can be $\tau$-tilting finite or infinite even though their defect groups are isomorphic (see \cite[Corollary 4]{EJR} for example). Instead of defect groups, we should consider so-called \textit{$p$-hyperfocal subgroups}. Denote by $O^p(G)$ the smallest normal subgroup of $G$ such that its quotient is a $p$-group.\\

\begin{dfn}
    A \textit{$p$-hyperfocal subgroup} of a finite group $G$ is the intersection of a Sylow $p$-subgroup and $O^p(G)$.\\
\end{dfn}

\begin{prop}\label{tfhyp}
    Let $R$ be a $p$-hyperfocal subgroup of a finite group $G$. Then a group algebra $kG$ is $\tau$-tilting finite if one of the following holds.
    \begin{enumerate}
        \setlength{\parskip}{0cm}
        \setlength{\itemsep}{0cm}
        \item $R$ is cyclic.
        \item $p=2$ and $R$ is isomorphic to a dihedral, semidihedral or generalized quaternion group.
    \end{enumerate}
\end{prop}
\begin{proof}
    By \cite[Theorem 3.10]{KK3}, if $kO^p(G)$ is $\tau$-tilting finite, then so is $kG$. Hence, the assertion holds from Theorems \ref{tameblock} and \ref{repblock} since $R$ is a Sylow $p$-subgroup of $O^p(G)$.
\end{proof}

\vspace{3mm}

We conjecture that the converse of Proposition \ref{tfhyp} is also true. In section 3, we will give a criterion for a group algebra $k[P\rtimes H]$ to be $\tau$-tilting finite, where $P$ is an abelian $p$-group and $H$ is an abelian $p'$-group acting on $P$, and verify that the converse of Proposition \ref{tfhyp} holds in this case.\\

\subsection{Zigzag cycles and $\tau$-tilting infiniteness}

Let $Q$ be a finite quiver and $I$ an admissible ideal of a path algebra $kQ$, that is, an ideal of $kQ$ satisfying $J^m_{kQ}\subseteq I\subseteq J^2_{kQ}$ for some integer $m\geq2$, where $J_{kQ}$ denotes the Jacobson radical of $kQ$. We denote by $Q_0$ (resp. $Q_1$) the set of vertices (resp. arrows) of $Q$ and by $s,t: Q_1\rightarrow Q_0$ maps sending arrows to their sources and targets, respectively.\\

\begin{dfn}\label{zigdfn}
    A \textit{zigzag cycle} of length $n$ is a sequence of arrows $a_1,\ldots,a_n\in Q_1$ satisfying the following conditions:
    \begin{enumerate}
        \setlength{\parskip}{0cm}
        \setlength{\itemsep}{0cm}
        \item A sequence $a_1,\ldots,a_n$ is a zigzag path, that is, for every $1\leq i\leq n-1$, $t(a_i)=t(a_{i+1})$ if $i$ is odd and $s(a_i)=s(a_{i+1})$ if $i$ is even.
        \item $s(a_1)=t(a_n)$ if $n$ is odd and $s(a_1)=s(a_n)$ if $n$ is even.\\
    \end{enumerate}
\end{dfn}

\noindent We can draw zigzag cycles as follows (even length on the left and odd length on the right):
$$
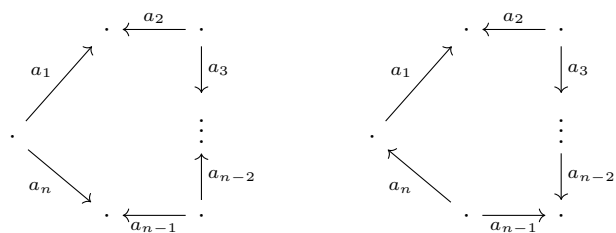
\begin{tikzcd}
     & \cdot & \cdot \lar["a_2"'] \dar["a_3"] &[10mm] & \cdot & \cdot \lar["a_2"'] \dar["a_3"] \\
    \cdot \ar[ru,"a_1"] \ar[rd,"a_n"'] & & \vdots & \cdot \ar[ru,"a_1"] & & \vdots \dar["a_{n-2}"] \\
     & \cdot & \cdot \lar["a_{n-1}"] \uar["a_{n-2}"'] &  & \cdot \ar[lu,"a_n"] \rar["a_{n-1}"'] & \cdot 
\end{tikzcd}
$$

\vspace{3mm}

\begin{dfn}
    We call each element $f\in I$ a \textit{relation}. For any relation $f\in I$, we can uniquely write $f$ as a linear combination of paths $q_i$, that is, $f=\sum_{i}c_iq_i$ ($c_i\in k^\times$). Then, for any path $q$, we say that \textit{a path $q$ appears in a relation $f$} if $q=q_i$ for some $i$.\\
\end{dfn}

\begin{prop}\label{zig}
    Assume that there exists a zigzag cycle $a_1,\ldots,a_n$ of length $n\geq2$ with distinct vertices in $Q$ and $I$ is generated by $f_1,\ldots,f_r\in kQ$. Then $kQ/I$ is $\tau$-tilting infinite if one of the following holds:
    \begin{enumerate}
        \setlength{\parskip}{0cm}
        \setlength{\itemsep}{0cm}
        \item The length $n$ is even.
        \item The length $n$ is odd and a path $a_1a_n$ does not appear in any relations $f_1,\ldots,f_r$.
    \end{enumerate}
\end{prop}
\begin{proof}
    Let $Q'$ be a subquiver of $Q$ consisting of all the vertices and arrows lying in the zigzag cycle $a_1,\ldots,a_n$, that is, $Q'_0:=\{s(a_i)\}_{1\leq i\leq n}\cup\{t(a_i)\}_{1\leq i\leq n}$ and $Q'_1:=\{a_i\}_{1\leq i\leq n}$, and $I'$ an ideal of $kQ$ generated by all the vertices and arrows not in $Q'$. By the assumption, every path appearing in $f_1,\ldots,f_r$ contains an arrow other than $a_1,\ldots,a_n$, and hence it follows that $I\subset I'$. Thus, we have an algebra surjection $kQ/I\twoheadrightarrow kQ/I'\cong kQ'$. Since $Q'$ is an extended Dynkin quiver, $kQ'$ has infinitely many isoclasses of indecomposable preprojective modules, which are bricks (for example, see \cite[Theorem 7.3.1 and Proposition 6.4.3]{DW} or \cite[Section VII.2]{ASS}), and hence we can obtain infinitely many isoclasses of bricks over $kQ/I$ via the surjection $kQ/I\twoheadrightarrow kQ'$. Therefore, $kQ/I$ is $\tau$-tilting infinite by Theorem \ref{DIJ}.
\end{proof}

\vspace{3mm}

\section{Main results}

Throughout this section, let $P$ be a $p$-group of rank $n$, $H$ a $p'$-group acting on $P$, and $G:=P\rtimes H$.\\
\indent First, we will determine the shape of the Gabriel quiver of $kG$. Notice that the Jacobson radical $J_{kG}$ of $kG$ is equal to $J_{kP}kH$ (see \cite[Theorem 1.1]{RR}) and that we have $kG/(J_{kP}kH)\cong kH$. Let $M:=J_{kP}/J^2_{kP}$. We can regard $M$ as an $n$-dimensional $kH$-module by conjugation. Denote the set of irreducible characters of $H$ by $\mathrm{Irr}\,H$, the dual character of $\chi\in\mathrm{Irr}\,H$ by $\chi^*$, and the simple $kH$-module corresponding to $\chi\in\mathrm{Irr}\,H$ by $S_\chi$.\\ 

\begin{prop}\label{genquiv}
    The Gabriel quiver of $kG$ is given by the following:
    \begin{itemize}
        \setlength{\parskip}{0cm}
        \setlength{\itemsep}{0cm}
        \item The vertex set is $\mathrm{Irr}\,H$.
        \item For $\lambda,\mu\in\mathrm{Irr}\,H$, the number of arrows from $\lambda$ to $\mu$ is $\mathrm{dim}_k\mathrm{Hom}_{kH}(S_\mu,M\otimes S_\lambda)$.
    \end{itemize}
\end{prop}
\begin{proof}
    The embedding functor $kH\mathrm{\textrm{-}mod}\hookrightarrow kG\mathrm{\textrm{-}mod}$ induced by the surjection $kG\twoheadrightarrow kG/J_{kG}\cong kH$ yields a bijection between the set of isoclasses of simple $kH$-modules and that of simple $kG$-modules. Therefore, we can identify the vertex set of the Gabriel quiver of $G$ with $\mathrm{Irr}\,H$.\\
    \indent For $\lambda,\mu\in\mathrm{Irr}\,H$, the number of arrows from $\lambda$ to $\mu$ is equal to $\mathrm{dim}_k\,e'_\mu(J_{kG}/J^2_{kG})e'_\lambda$, where we take primitive idempotents $e'_\lambda$ and $e'_\mu$ such that $kHe'_\lambda\cong S_\lambda$ and $kHe'_\mu\cong S_\mu$. Note that $kGe'_\mu$ is the projective cover of and $S_\mu$ in $kG\textrm{-mod}$. Since $(J_{kG}/J^2_{kG})e'_\lambda=(J_{kP}/J^2_{kP})kHe'_\lambda$, we have
    \begin{align*}
        e'_\mu(J_{kG}/J^2_{kG})e'_\lambda & \cong\mathrm{Hom}_{kG}(kGe'_\mu,(J_{kG}/J^2_{kG})e'_\lambda)\\
        & =\mathrm{Hom}_{kG}(\mathrm{Ind}^G_H\,S_\mu,(J_{kP}/J^2_{kP})kHe'_\lambda)\\
        & \cong\mathrm{Hom}_{kH}(S_\mu,\mathrm{Res}^G_H\,(J_{kP}/J^2_{kP})kHe'_\lambda)\\
        & \cong\mathrm{Hom}_{kH}(S_\mu,M\otimes S_\lambda).
    \end{align*}
    Therefore, the assertion follows.
\end{proof}

\vspace{3mm}

\begin{ex}
    Let $p>3$, $P:=(C_p)^3$ with generators $a,b,c$, and $H:=\mathfrak{S}_3$. Recall that $H$ has three irreducible representations: the trivial, standard and sign representation. Denote the corresponding characters by $\mathrm{triv}, \mathrm{std}$ and $\mathrm{sgn}$, respectively. We define $G:=P\rtimes H$ via the action of $H$ on $P$ given by the permutation of the entries. Then $M$ has a $k$-basis
    $$\{1-a+J^2_{kP},\,1-b+J^2_{kP},\,1-c+J^2_{kP}\}$$
    and is isomorphic to $S_\mathrm{triv}\oplus S_\mathrm{std}$ as a $kH$-module. Moreover, we have the following isomorphisms as $kH$-modules:
    $$S_\mathrm{triv}\otimes S_\chi\cong S_\chi\;(\forall\chi\in\mathrm{Irr}\,H),$$
    $$S_\mathrm{std}\otimes S_\mathrm{std}\cong S_\mathrm{triv}\oplus S_\mathrm{std}\oplus S_\mathrm{sgn},\,S_\mathrm{std}\otimes S_\mathrm{sgn}\cong S_\mathrm{std},\,S_\mathrm{sgn}\otimes S_\mathrm{sgn}\cong S_\mathrm{triv}.$$
    Therefore, by Proposition \ref{genquiv}, the Gabriel quiver of $kG$ is as follows:
    $$
    \begin{tikzcd}
        \mathrm{triv} \ar[loop left] \rar[yshift=0.5ex] & \mathrm{std} \ar[out=120,in=60,loop,looseness=5] \ar[out=-60,in=-120,loop,looseness=5] \lar[yshift=-0.5ex] \rar[yshift=0.5ex] & \mathrm{sgn} \lar[yshift=-0.5ex] \ar[loop right] & .
    \end{tikzcd}
    $$
\end{ex}

\vspace{4mm}

\indent Assume that $P$ and $H$ are both abelian. Note that $kG$ is a basic algebra since $H$ is abelian. In this situation, Benson, Kessar and Linckelmann \cite{BKL1} described the quiver and relations of a twisted group algebra $k_\beta G$ for any 2-cocycle $\beta\in H^2(H,k^\times)$, but we shall give the proof in case $\beta=1$ for convenience.\\
\indent By \cite[Theorem 2.2 in Chapter 5]{G}, we can decompose $P$ into $H$-invariant homocyclic summands, that is, we can assume that $P=\prod_{i\geq1} (C_{p^i})^{t_i}$ and $h(C_{p^i})^{t_i}h^{-1}=(C_{p^i})^{t_i}$ for every $h\in H$ and $i\geq1$. Let $P_i:=(C_{p^{i}})^{t_{i}}\leq P$,
$$\Gamma:=\{(i,j)\mid i\geq1,\,t_i\neq0,\,1\leq j\leq t_i\},$$
and $\pi:kP\twoheadrightarrow kP/J^2_{kP}$ the natural surjection. Since $kH$ is semisimple, we can take a $H$-invariant complement $W_i$ of $J^2_{kP_i}$ in $J_{kP_i}$. Notice that $J_{kP}$ has a decomposition $J_{kP}=\bigoplus_{i\geq1}W_i\oplus J_{kP_i}^2$ as a $kH$-module.\\

\begin{prop}\label{abelquiv}
    Assume that $P$ and $H$ are both abelian. Then there exists a set of elements $\{m_{ij}\in W_i\}_{(i,j)\in\Gamma}$ such that $M=\bigoplus_{(i,j)\in\Gamma}kH\cdot\pi(m_{ij})$ and each $kH\cdot\pi(m_{ij})$ is simple as a $kH$-module, where the dots mean the conjugation action of $H$. Moreover, $kG$ is isomorphic to $kQ/I$, where a quiver $Q$ and an admissible ideal $I$ are given by the following:
    \begin{itemize}
        \setlength{\parskip}{0cm}
        \setlength{\itemsep}{0cm}
        \item The vertex set of $Q$ is $\mathrm{Irr}\,H$.
        \item The arrow set of $Q$ is $\{\alpha_{ij\lambda}:\lambda\rightarrow\chi_{ij}\otimes\lambda\}_{(i,j)\in\Gamma,\,\lambda\in\mathrm{Irr}\,H}$, where $\chi_{ij}\in\mathrm{Irr}\,H$ satisfies $kH\cdot\pi(m_{ij})\cong S_{\chi_{ij}}$.
        \item $I=\langle \alpha_{ij}\alpha_{i'j'}-\alpha_{i'j'}\alpha_{ij},\,\alpha_{ij}^{p^i}\mid (i,j),(i',j')\in\Gamma\rangle$, where the relations mean that the following equations hold for all $\lambda\in\mathrm{Irr}\,H$:
    \end{itemize}
    $$
    \begin{tikzcd}
        (\lambda \rar["\alpha_{i'j'\lambda}"] & \chi_{i'j'}\otimes\lambda \rar["\alpha_{ij,\chi_{i'j'}\otimes\lambda}"] &[8mm] \chi_{ij}\otimes\chi_{i'j'}\otimes\lambda) = (\lambda \rar["\alpha_{ij\lambda}"] & \chi_{ij}\otimes\lambda \rar["\alpha_{i'j',\chi_{ij}\otimes\lambda}"] &[8mm] \chi_{i'j'}\otimes\chi_{ij}\otimes\lambda),
        \end{tikzcd}
    $$
    $$
    \begin{tikzcd}
        (\lambda \rar["\alpha_{ij\lambda}"] & \chi_{ij}\otimes\lambda \rar["\alpha_{ij,\chi_{ij}\otimes\lambda}"] &[8mm] \chi_{ij}^{\otimes 2}\otimes\lambda \rar["\alpha_{ij,\chi_{ij}^{\otimes 2}\otimes\lambda}"] &[8mm] \cdots \rar & \chi_{ij}^{\otimes p^i}\otimes\lambda)=0.
    \end{tikzcd}
    $$
\end{prop}
\begin{proof}
    Let $\{i\geq1\mid t_i\neq0\}=\{i_1<i_2<\cdots\}$. We can take a set of generators $\{g_{ij}\}_{(i,j)\in\Gamma}$ of $P$ such that the order of $g_{ij}$ is $p^i$ and that the action of $h\in H$ on $P$ is represented by the following block diagonal $n\times n$ integer matrix with respect to the generators $g_{ij}$:
    $$
    A^h=
    \begin{pmatrix}
        A^h_{i_1} & 0 & \cdots  \\
        0 & A^h_{i_2} & \cdots  \\
        \vdots & \vdots & \ddots \\
    \end{pmatrix}
    \in\mathrm{End}(P)=
    \begin{pmatrix}
        \mathrm{End}((C_{p^{i_1}})^{t_{i_1}}) & \mathrm{Hom}((C_{p^{i_2}})^{t_{i_2}},(C_{p^{i_1}})^{t_{i_1}}) & \cdots  \\
        \mathrm{Hom}((C_{p^{i_1}})^{t_{i_1}},(C_{p^{i_2}})^{t_{i_2}}) & \mathrm{End}((C_{p^{i_2}})^{t_{i_2}}) & \cdots  \\
        \vdots & \vdots & \ddots \\
    \end{pmatrix}.
    $$
    Since $M$ has a $k$-basis
    $$\mathcal{B}:=\{1-g_{ij}+J^2_{kP}\mid(i,j)\in\Gamma\}$$
    and $1-g_{ij}g_{i'j'}=(1-g_{ij})+(1-g_{i'j'})-(1-g_{ij})(1-g_{i'j'})\in (1-g_{ij})+(1-g_{i'j'})+J^2_{kP}$, the action of $h\in H$ on $M$ is represented by the following block diagonal $n\times n$-matrix with respect to $\mathcal{B}$:
    $$
    \overline{A^h}=
    \begin{pmatrix}
        \overline{A^h_{i_1}} & 0 & \cdots \\
        0 & \overline{A^h_{i_2}} & \cdots \\
        \vdots & \vdots & \ddots \\
    \end{pmatrix}
    \in GL_n(\mathbb{F}_p)\subset GL_n(k)\cong GL(M),
    $$
    where $\overline{(-)}$ denotes a matrix whose entries are considered modulo $p$. Since $kH$ is semisimple and $H$ is abelian, the matrices $\{\overline{A^h}\}_{h\in H}$ are simultaneously diagonalizable. Thus, we can take a set of elements $\{m_{ij}\in W_i\}_{(i,j)\in\Gamma}$ such that $M=\bigoplus_{(i,j)\in\Gamma}kH\cdot\pi(m_{ij})$ as a $kH$-module, each $kH\cdot\pi(m_{ij})$ is a simple $kH$-module and $\bigoplus_{1\leq j\leq t_i}kH\cdot\pi(m_{ij})$ is isomorphic to a $kH$-module defined by a group homomorphism $\overline{A_i}:H\ni h\mapsto\overline{A^h_i}\in GL_{t_i}(\mathbb{F}_p)$.\\
    \indent Take $\chi_{ij}\in\mathrm{Irr}\,H$ such that $kH\cdot\pi(m_{ij})\cong S_{\chi_{ij}}$. For $\chi\in\mathrm{Irr}\,H$, let $e_\chi:=|H|^{-1}\sum_{h\in H}\chi(h^{-1})h$ be the primitive idempotent corresponding to $\chi$. Then we have
    \begin{align*}
        e_\mu\pi(m_{ij})e_\lambda & = \frac{1}{|H|}\sum_{h\in H}\mu(h^{-1})(h\pi(m_{ij})h^{-1})he_\lambda\\
        & =\frac{1}{|H|}\sum_{h\in H}\mu(h^{-1})\chi_{ij}(h)\pi(m_{ij})\lambda(h)e_\lambda\\
        & =\frac{1}{|H|}\sum_{h\in H}(\mu\otimes\chi^*_{ij}\otimes\lambda^*)(h^{-1})\pi(m_{ij})e_\lambda\\
        & =\left\{
        \begin{array}{ll}
            \pi(m_{ij})e_\lambda & (\mu=\chi_{ij}\otimes\lambda), \\
            0 & (\mu\neq\chi_{ij}\otimes\lambda).
        \end{array}
        \right.
    \end{align*}
    Since a set $\{m_{ij}\}_{1\leq j\leq t_i}$ is a $k$-basis of $W_i$ and $J_{kP_i}=W_i\oplus J^2_{kP_i}$ as a $kH$-module, we have
    $$
    e_\mu m_{ij}e_\lambda=
    \left\{
        \begin{array}{ll}
            m_{ij}e_\lambda & (\mu=\chi_{ij}\otimes\lambda), \\
            0 & (\mu\neq\chi_{ij}\otimes\lambda).
        \end{array}
    \right.
    $$
    Hence, the set of arrows emanating from a vertex $\lambda\in\mathrm{Irr}\,H$ is
    $$\{\alpha_{ij\lambda}:=e_{\chi_{ij}\otimes\lambda}m_{ij}e_\lambda=m_{ij}e_\lambda:\lambda\rightarrow\chi_{ij}\otimes\lambda\}_{(i,j)\in\Gamma}.$$
    Since $m_{ij}\in J_{kP_i}$ and so $m_{ij}^{p^i}=0$, the following equations hold: 
    \begin{align*}
        e_{\chi_{ij}\otimes\chi_{i'j'}\otimes\lambda}m_{ij}e_{\chi_{i'j'}\otimes\lambda}m_{i'j'}e_\lambda & = m_{ij}m_{i'j'}e_\lambda \\
        & = m_{i'j'}m_{ij}e_\lambda \\
        & = e_{\chi_{i'j'}\otimes\chi_{ij}\otimes\lambda}m_{i'j'}e_{\chi_{ij}\otimes\lambda}m_{ij}e_\lambda,
    \end{align*}
    $$e_{\chi_{ij}^{\otimes p^i}\otimes\lambda}m_{ij}\cdots e_{\chi_{ij}^{\otimes 2}\otimes\lambda}m_{ij}e_{\chi_{ij}\otimes\lambda}m_{ij}e_\lambda=m^{p^i}_{ij}e_\lambda=0.$$
    Thus, there exists a surjective algebra homomorphism $kQ/I\twoheadrightarrow kG$. However, we can easily see that the dimensions of $kQ/I$ and $kG$ are the same since the number of equivalence classes of paths in $kQ/I$ emanating from each vertex is $|P|$ by the above relations. Therefore, $kG$ is isomorphic to $kQ/I$.
\end{proof}

\vspace{3mm}

\begin{ex}
    Let $p=3$, $P:=(C_3)^2\times C_9$ with generators $a,b\in (C_3)^2,\,c\in C_9$, and $H:=C_4$ with a generator $d$. Recall that $H$ has four irreducible characters $\chi_i$ ($i=0,1,2,3$) sending $d$ to $\zeta^i$, where $\zeta$ denotes a primitive fourth root of unity in $k$. We define $G:=P\rtimes H$ via the following action of $H$ on $P$:
    $$dad^{-1}=ab,\,dbd^{-1}=ab^2,\,dcd^{-1}=c^8,$$
    which is represented by the following matrix with respect to the generators $a,b,c$:
    $$
    A^d=
    \begin{pmatrix}
        1 & 1 & 0  \\
        1 & 2 & 0  \\
        0 & 0 & 8 \\
    \end{pmatrix}.
    $$
    Hence, as in the proof of Proposition \ref{abelquiv}, the action of $d\in H$ on $M$ is represented by the following matrix with respect to a $k$-basis $\{1-a+J^2_{kP},\,1-b+J^2_{kP},\,1-c+J^2_{kP}\}$:
    $$
    \overline{A^d}=
    \begin{pmatrix}
        1 & 1 & 0  \\
        1 & 2 & 0  \\
        0 & 0 & 2 \\
    \end{pmatrix},
    $$
    which is congruent over $k$ to the matrix
    $$
    \begin{pmatrix}
        \zeta & 0 & 0  \\
        0 & \zeta^3 & 0  \\
        0 & 0 & \zeta^2 \\
    \end{pmatrix}.
    $$
    Therefore, by Proposition \ref{abelquiv}, $kG$ is isomorphic to $kQ/I$, where
    $$
    \begin{tikzcd}
         & \chi_0 \ar["\alpha_{10}",dd,shift left=0.5ex] \ar[near start,"\alpha_{20}" sloped,rrdd,shift left=0.5ex] \ar["\alpha_{30}",rr,shift left=0.5ex]  & & \chi_3 \ar["\alpha_{13}",ll,shift left=0.5ex] \ar[near start,"\alpha_{23}"' sloped,lldd,shift left=0.5ex] \ar["\alpha_{33}",dd,shift left=0.5ex] & \\
        Q:= &  &  &  & \\
         & \chi_1 \ar["\alpha_{11}",rr,shift left=0.5ex] \ar[near start,"\alpha_{21}" sloped,rruu,shift left=0.5ex] \ar["\alpha_{31}",uu,shift left=0.5ex] &  & \chi_2 \ar["\alpha_{12}",uu,shift left=0.5ex] \ar[near start,"\alpha_{22}"' sloped,lluu,shift left=0.5ex] \ar["\alpha_{32}",ll,shift left=0.5ex] & ,\\
    \end{tikzcd}
    $$
    $$I:=\langle\alpha_{j,i+l}\alpha_{il}-\alpha_{i,j+l}\alpha_{jl},\,\alpha_{1l}^3,\,\alpha_{3l}^3,\,\alpha_{2l}^9\mid i,j,l\in\mathbb{Z}/4\mathbb{Z}\rangle.$$
\end{ex}

\vspace{4mm}

\indent Our aim is to give the condition for $kG$ to be $\tau$-tilting finite when $P$ and $H$ are both abelian.\\

\begin{prop}[{\cite[(8.5)-(6), (24.4) and (24.6)]{Asch}}]\label{Asch}
    It holds that $[P,H]\trianglelefteq G$ and $P=C_P(H)[P,H]$. Moreover, if $P$ is abelian, then $P=C_P(H)\times[P,H]$. \\
\end{prop}

\begin{cor}\label{red}
    If $P$ is abelian, then there exists a poset isomorphism
    $$\mathrm{s\tau\textrm{-}tilt}\,kG\cong\mathrm{s\tau\textrm{-}tilt}\,k([P,H]\rtimes H).$$
\end{cor}
\begin{proof}
    The assertion follows from Proposition \ref{Asch} and \cite[Theorem 3.7]{H} because $C_P(H)$ is a central $p$-subgroup of $G$.
\end{proof}

\vspace{3mm}

\indent It follows from Proposition \ref{Asch} and Corollary \ref{red} that when we consider $\tau$-tilting finiteness of $kG$, we can assume that the centralizer $C_P(H)$ of $H$ in $P$ is trivial by replacing $P$ with $[P,H]$.\\

\begin{lem}\label{noloop}
    If $H$ is abelian and $C_P(H)=1$, then $\mathrm{Ext}_{kG}^1(S_\chi,S_\chi)=0$ for all $\chi\in\mathrm{Irr}\,H$. In particular, the Gabriel quiver of $kG$ has no loops.
\end{lem}
\begin{proof}
    Since $H$ is abelian, we have $\mathrm{Ext}_{kG}^1(S_\chi,S_\chi)\cong\mathrm{Ext}_{kG}^1(k_G,S^*_\chi\otimes S_\chi)\cong\mathrm{Ext}_{kG}^1(k_G,k_G)$, where $k_G$ denotes the trivial $kG$-module. Hence, by \cite[Corollary 10.13]{L}, it is sufficient to show that there exist no normal subgroups of $G$ with index $p$. \\
    \indent Assume that there exists a normal subgroup $N$ of $G$ with index $p$. By Proposition \ref{Asch} and the assumption that $C_P(H)=1$, we have $[P,H]=[H,P]=P$. Since $[P,P]\leq P$ and $[H,H]=1$, it follows that $[G,G]=P$. Hence, we have $N\geq [G,G]=P$ because $G/N$ is abelian. Therefore, it follows that $p=(G:N)\mid (G:P)=|H|$. However, this contradicts the assumption that $H$ is a $p'$-group.
\end{proof}

\vspace{3mm}

\begin{lem}\label{p2}
    Assume that $p=2$, $P$ and $H$ are abelian, and $C_P(H)=1$. Then $t_i\neq1$ for all $i\geq1$.
\end{lem}
\begin{proof}
    If there exists $i'\geq1$ such that $t_{i'}=1$, then $kH\cdot\pi(m_{i'1})\in\mathbb{F}_pH\textrm{-mod}$ in the proof of Proposition \ref{abelquiv} must be the trivial $kH$-module because $p=2$. Hence, by Proposition \ref{abelquiv}, the Gabriel quiver of $kG$ has loops $\{\alpha_{i'1\lambda}\}_{\lambda\in\mathrm{Irr}\,H}$, which contradicts Lemma \ref{noloop}.
\end{proof}

\vspace{3mm}

\indent Surprisingly, $\tau$-tilting finiteness of $kG$ is determined by the $p$-hyperfocal subgroup of $G$, as we will see in Theorem \ref{mainthm}.\\

\begin{lem}\label{hyp}
    It holds that $[P,H]$ is the $p$-hyperfocal subgroup of $G$.
\end{lem}
\begin{proof}
    For any $x\in P$ and $y\in H$, we have $xyx^{-1}y^{-1}\in O^p(G)$ since $xyx^{-1}$ and $y^{-1}$ are $p'$-elements (recall that $O^p(G)$ is generated by all $p'$-elements of $G$). Moreover, we also have $H\leq O^p(G)$ since $H$ is a $p'$-group. Hence, $[P,H]\rtimes H\leq O^p(G)$. By Proposition \ref{Asch}, $[P,H]\rtimes H$ is a normal subgroup of $G$ of $p$-power index. Thus, $O^p(G)\leq [P,H]\rtimes H$. Therefore, the $p$-hyperfocal subgroup of $G$ is $P\cap O^p(G)=P\cap ([P,H]\rtimes H)=[P,H].$
\end{proof}

\vspace{3mm}

\begin{thm}\label{mainthm}
    Let $P$ be an abelian $p$-group, $H$ an abelian $p'$-subgroup acting on $P$ and $G:=P\rtimes H$. Denote by $R$ the $p$-hyperfocal subgroup of $G$. Then $kG$ is $\tau$-tilting finite if and only if one of the following holds:
    \begin{enumerate}
        \setlength{\parskip}{0cm}
        \setlength{\itemsep}{0cm}
        \item $p=2$ and $R$ is trivial or isomorphic to $C_2\times C_2$.
        \item $p\geq3$ and $R$ is trivial or cyclic.
    \end{enumerate}
\end{thm}
\begin{proof}
    From Proposition \ref{Asch}, Corollary \ref{red} and Lemma \ref{hyp}, we can assume that $C_P(H)=1$ and $P=R$ by replacing $P$ with $[P,H]$.\\
    \indent In the case of (a) and (b), the group algebra $kG$ is $\tau$-tilting finite by Proposition \ref{tfhyp}.\\
    \indent For cases other than (a) and (b), we show $\tau$-tilting infiniteness of $kG$ by identifying $kG$ with $kQ/I$ in the sense of Proposition \ref{abelquiv} and applying Proposition \ref{zig}. Note that when $p=2$, it holds that $t_i\neq1$ for all $i\geq1$ by Lemma \ref{p2}, and hence it suffices to consider the following cases (i)-(iii).\\
    (i) If $p\geq3$ and $P$ has rank $\geq2$, then we can choose two distinct pairs $(i,j),(i',j')\in\Gamma$. By considering the following zigzag path starting from arbitrary $\lambda\in\mathrm{Irr}\,H$
    $$
    \begin{tikzcd}
        \lambda \rar["\alpha_{ij}"] & \chi_{ij}\otimes\lambda & \chi^*_{i'j'}\otimes\chi_{ij}\otimes\lambda \lar["\alpha_{i'j'}"'] \rar["\alpha_{ij}"] & \cdots,
    \end{tikzcd}
    $$
    we obtain a zigzag cycle of length $\geq2$ with distinct vertices because $Q$ has no loops by Lemma \ref{noloop}. Since neither $\alpha^2_{ij}$ nor $\alpha^2_{i'j'}$ appears in the relations as in Proposition \ref{abelquiv}, $kG$ is $\tau$-tilting infinite by Proposition \ref{zig}.\\
    (ii) If $p=2$ and $t_i\geq2$ for some $i\geq2$, then, by taking distinct pairs $(i,j),(i,j')\in\Gamma$ with $i\geq2$, we can show that $kG$ is $\tau$-tilting infinite by the same argument as in (i). Note that $\alpha_{1j}^2=0$ for any $1\leq j\leq t_1$ by Proposition \ref{abelquiv}, and hence the assumption $i\geq 2$ is crucial.\\
    (iii) If $p=2$, $t_1\geq3$ and $t_i=0$ for all $i\geq2$, then we should consider the action of the Frobenius map on the $kH$-module $M=J_{kP}/J^2_{kP}\cong S_{\chi_{11}}\oplus\cdots\oplus S_{\chi_{1t_1}}$. Since the representation $M$ of $H$ is defined over $\mathbb{F}_2$ (see the proof of Proposition \ref{abelquiv}), the Frobenius map $F:k\ni a\mapsto a^2\in k$ permutes $\chi_{11},\ldots,\chi_{1t_1}$. In other words, there exists $\sigma\in\mathfrak{S}_{t_1}$ such that $F\circ\chi_{1j}=\chi_{1\sigma(j)}$ for all $1\leq j\leq t_1$. Note that it follows that $F\circ\chi_{1j}=\chi^{\otimes 2}_{1j}$ for $1\leq j\leq t_1$. Consider the $F$-orbits under the action on $\{\chi_{11},\ldots,\chi_{1t_1}\}$. \\
    \indent If there exists an $F$-orbit $\{\chi\}$ of size $1$, then $F\circ\chi=\chi^{\otimes 2}=\chi$ implies that $\chi$ is the trivial character, which contradicts the fact that $Q$ has no loops by Lemma \ref{noloop}. Therefore, every $F$-orbit has size $\geq 2$.\\
    \indent If there exists an $F$-orbit $\mathfrak{O}$ of size $\geq3$, then we can take three distinct characters $\chi,\chi^{\otimes2},\chi^{\otimes4}\in\mathfrak{O}$ and consider the following zigzag cycle $\gamma$ starting from arbitrary $\lambda\in\mathrm{Irr}\,H$
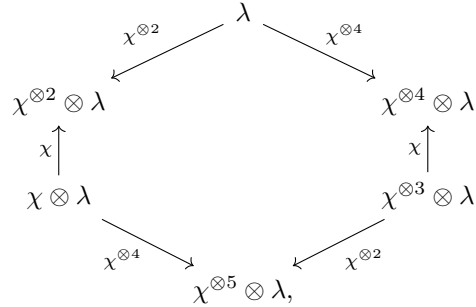
    $$
    \begin{tikzcd}
         & \lambda \ar[ld,"\chi^{\otimes2}"'] \ar[rd,"\chi^{\otimes4}"] & \\
        \chi^{\otimes2}\otimes\lambda & & \chi^{\otimes4}\otimes\lambda \\
         \chi\otimes\lambda \ar[u,"\chi"]  \ar[rd,"\chi^{\otimes4}"'] & & \chi^{\otimes3}\otimes\lambda \ar[u,"\chi"]  \ar[ld,"\chi^{\otimes2}"] \\
         & \chi^{\otimes5}\otimes\lambda, &
    \end{tikzcd}
    $$
    where each arrow is labeled by the corresponding character. Since $\mathfrak{O}$ does not contain the trivial character, all the vertices in $\gamma$ are distinct unless $\chi^{\otimes3}$ or $\chi^{\otimes5}$ is the trivial character. Since $\chi\neq\chi^{\otimes4}$, $\chi^{\otimes3}$ cannot be the trivial character. If $\chi^{\otimes5}$ is the trivial character, then we have the following zigzag cycle with distinct vertices
    $$
    
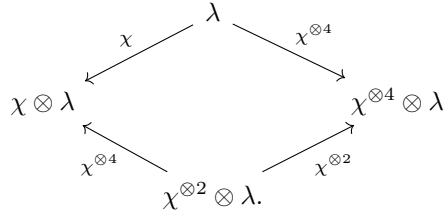
\begin{tikzcd}
         & \lambda \ar[ld,"\chi"'] \ar[rd,"\chi^{\otimes4}"] & \\
        \chi\otimes\lambda & & \chi^{\otimes4}\otimes\lambda \\
         & \chi^{\otimes2}\otimes\lambda. \ar[lu,"\chi^{\otimes4}"] \ar[ru,"\chi^{\otimes2}"'] &
    \end{tikzcd}
    $$
    Therefore, if there exists an $F$-orbit $\mathfrak{O}$ of size $\geq3$, then we can show that $kG$ is $\tau$-tilting infinite by considering the above zigzag cycles of even length with distinct vertices by Proposition \ref{zig}.\\
    \indent If the size of every $F$-orbit is 2, then we can take four distinct characters $\chi,\chi^{\otimes2},\chi ',\chi '^{\otimes2}\in\{\chi_{11},\ldots,\chi_{1t_1}\}$. Since $\chi^{\otimes3}$ and $\chi '^{\otimes3}$ are the trivial characters and $\chi '\neq\chi^{\otimes2}=\chi^*$, we have the following zigzag cycle with distinct vertices
    $$
    
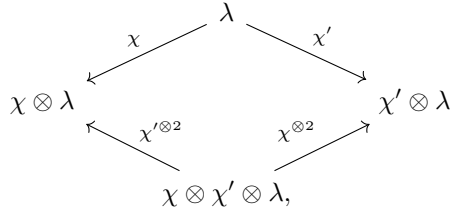
\begin{tikzcd}
         & \lambda \ar[ld,"\chi"'] \ar[rd,"\chi '"] & \\
        \chi\otimes\lambda & & \chi '\otimes\lambda \\
         & \chi\otimes\chi '\otimes\lambda, \ar[lu,"\chi '^{\otimes2}"'] \ar[ru,"\chi^{\otimes2}"] &
    \end{tikzcd}
    $$
    which shows that $kG$ is $\tau$-tilting infinite by Proposition \ref{zig}.
\end{proof}

\vspace{3mm}

\indent When $P$ is nonabelian, the problem of determining whether $kG$ is $\tau$-tilting finite cannot be reduced to the case $C_P(H)=1$ since $C_P(H)$ is not necessarily normal in $G$. Nonetheless, if $C_P(H)$ is contained in the Frattini subgroup $\Phi(P)$ of $P$, then $\tau$-tilting finiteness of $kG$ is determined by the $p$-hyperfocal subgroup of $G$ except in the case when $p=2$ and $P$ has rank 2.\\

\begin{cor}\label{subcor}
    Let $P$ be a $p$-group, $H$ an abelian $p'$-subgroup acting on $P$ and $G:=P\rtimes H$. Assume that $C_P(H)\leq \Phi(P)$. Then the following hold:
    \begin{enumerate}
        \setlength{\parskip}{0cm}
        \setlength{\itemsep}{0cm}
        \item If $p=2$ and $kG$ is $\tau$-tilting finite, then $P$ is trivial or has rank $2$.
        \item If $p\geq3$ and $kG$ is $\tau$-tilting finite, then $P$ has rank $\leq1$.
    \end{enumerate}
\end{cor}
\begin{proof}
    (a) Consider the surjection of group algebras $kG\twoheadrightarrow k[G/\Phi(P)]=k[(P/\Phi(P))\rtimes H]$. By \cite[(18.7)-(4)]{Asch} and the assumption, we have $C_{P/\Phi(P)}(H)=C_P(H)\Phi(P)/\Phi(P)=1$. Hence, the $p$-hyperfocal subgroup of $(P/\Phi(P))\rtimes H$ is $P/\Phi(P)$ by Proposition \ref{Asch} and Lemma \ref{hyp}. Thus, if $kG$ is $\tau$-tilting finite, then so is $k[(P/\Phi(P))\rtimes H]$, which implies that $P/\Phi(P)$ is trivial or isomorphic to $C_2\times C_2$ by Theorem \ref{mainthm}. Since $P/\Phi(P)$ has the same rank as $P$, the assertion follows.\\
    (b) follows from the same argument as in (a).
\end{proof}

\vspace{3mm}

\begin{cor}\label{maincor}
    Let $P$ be a $p$-group, $H$ an abelian $p'$-subgroup acting on $P$, $G:=P\rtimes H$ and $R$ the $p$-hyperfocal subgroup of $G$. Assume that $C_P(H)\leq \Phi(P)$. Then the following hold:
    \begin{enumerate}
        \setlength{\parskip}{0cm}
        \setlength{\itemsep}{0cm}
        \item When $p=2$ and $P$ has rank $\neq2$, $kG$ is $\tau$-tilting finite if and only if $R$ is trivial.
        \item When $p\geq3$, $kG$ is $\tau$-tilting finite if and only if $R$ has rank $\leq1$.
    \end{enumerate}
\end{cor}
\begin{proof}
    The assertions follow immediately from Proposition \ref{tfhyp} and Corollary \ref{subcor}.
\end{proof}

\vspace{3mm}

\begin{rem}\label{p2rank2}
    Notice that we can replace $R$ with $P$ in Corollary \ref{maincor}, which says that $\tau$-tilting finiteness of $kG$ is determined by the rank of $P$ (or $R$) except in the case when $p=2$ and $P$ has rank 2. However, even if $P$ is abelian, $kG$ can be $\tau$-tilting finite or infinite when $p=2$ and $P$ has rank 2. For example, let $p=2$,
    $$G_1:=(C_2\times C_2)\rtimes C_3=\langle a,b,c\mid a^2=b^2=c^3=1,ab=ba,cac^{-1}=b,cbc^{-1}=ab\rangle,\,\mathrm{and}$$
    $$G_2:=(C_4\times C_4)\rtimes C_3=\langle a,b,c\mid a^4=b^4=c^3=1,ab=ba,cac^{-1}=b,cbc^{-1}=ab\rangle.$$
    We can easily check that both groups are centerless, and hence the $p$-hyperfocal subgroups of $G_1$ and $G_2$ are $C_2\times C_2$ and $C_4\times C_4$, respectively, by Proposition \ref{Asch} and Lemma \ref{hyp}. Thus, by Theorem \ref{mainthm}, $kG_1$ is $\tau$-tilting finite but $kG_2$ is $\tau$-tilting infinite.\\
\end{rem}

\section*{Acknowledgements}
The authors would like to thank Shigeo Koshitani for useful discussions. The first author was supported by JST SPRING, Grant Number JPMJSP2110.\\

\end{document}